\documentclass[12 pt]{amsart}
\usepackage{amscd,amssymb,amsmath,amsthm}
\input xy
\xyoption{all}
\newtheorem{Proposition}{Proposition}
\newtheorem{Question}{Question}
\newtheorem{Lemma}{Lemma}
\newtheorem{Theorem}{Theorem}

\newcommand{\proj}{\mathbb{P}}

\newcommand{\rarr}{\rightarrow}
\newcommand{\oh}{{\mathcal{O}}}
\newcommand{\com}{\mathbb{C}}
\newcommand{\Q}{\mathbb{Q}}

\newcommand{\lan}{\langle}
\newcommand{\ran}{\rangle}

\newcommand{\nn}{m}

\def\scup{\mathbin{\text{\scriptsize$\cup$}}}

\newcommand{\hodge}{\mathbb{E}}

\begin{document}
\baselineskip=16pt

\title{ The $\kappa$ ring of
the moduli of curves of compact type: II}
\author{R. Pandharipande}
\date{June 2009}

\begin{abstract}
The subalgebra of the tautological ring
of the moduli of 
curves of compact type
generated by the $\kappa$ classes is
studied. Relations, constructed 
via the virtual geometry of the moduli 
of stable maps, are used to prove
universality results
relating the $\kappa$ rings
in genus 0 to  higher genus.
Predictions  for $\kappa$  classes of the Gorenstein conjecture
are proven. 
\end{abstract}

\maketitle
\setcounter{tocdepth}{1}
\tableofcontents

\section{Introduction}

\subsection{$\kappa$ classes}
Let $\overline{M}_{g,n}$ be the moduli space of
genus $g$, $n$-points stable curves.
The $\kappa$ classes in the Chow ring 
$A^*(\overline{M}_{g,n})$ with $\mathbb{Q}$-coefficients
are defined by the following construction.
Let
$$\epsilon: \overline{M}_{g,n+1} \rarr \overline{M}_{g,n}$$
be the universal curve viewed as the ($n+1$)-pointed space, let
$$\mathbb{L}_{n+1} \rarr \overline{M}_{g,n+1}$$
be the line bundle obtained from the cotangent space
of the last marking, 
and let
$$\psi_{n+1} = c_1(\mathbb{L}_{n+1})\ \in A^1(\overline{M}_{g,n+1})$$
be the Chern class.
The $\kappa$ classes
are
$$\kappa_i = \epsilon_*(\psi^{i+1}_{n+1})\  \in A^i(\overline{M}_{g,n}), 
\ \ \ i \geq 0\ .$$
The simplest is $\kappa_0$ which equals $2g-2+n$ times the unit
in $A^0(\overline{M}_{g,n})$.

The $\kappa$ classes on the moduli space of 
curves of compact type
$$M_{g,n}^c\subset \overline{M}_{g,n}$$
are defined by restriction.
The $\kappa$ ring 
$$
\kappa^*(M_{g,n}^c) \subset   A^*(M_{g,n}^c), 
$$
is the $\mathbb{Q}$-subalgebra
generated by the $\kappa$ classes.
The $\kappa$ rings are graded by degree.

By the results of \cite{kap1},
$\kappa^*(M_{g,n}^c)$ is generated as a $\mathbb{Q}$-algebra
by 
$$\kappa_1,\kappa_2,\ldots, \kappa_{g-1+\lfloor\frac{n}{2}
\rfloor}.$$
Moreover, there are no relation of degree less than or
equal to ${g-1+\lfloor\frac{n}{2}
\rfloor}$ if $n>0$.

\subsection{Universality}
Let $x_1,x_2,x_3, \ldots $ be variables with $x_i$ of degree $i$, and let 
$$f\in \mathbb{Q}[x_1,x_2,x_3,\ldots]$$
be {\em any} graded homogeneous polynomial.
The following universality property was stated in \cite{kap1}.

\begin{Theorem}\label{mmmj}
If
$f(\kappa_i) = 0 \in \kappa^*(M_{0,n}^c)$, then
$$f(\kappa_i) = 0 \in \kappa^*(M_{g,n-2g}^c)$$
for all genera $g$ for which $n-2g\geq 0$.
\end{Theorem}

By Theorem 1, the higher genus $\kappa$ rings are
canonically quotients of the genus 0 rings,
$$\kappa^*(M_{0,2g+n}^c) \stackrel{\iota_{g,n}}{\rarr} 
\kappa^*(M^c_{g,n}) \rarr 0 .$$
Theorem 1 is our main result here.

\subsection{Bases}
Let $P(d)$ be the set of partitions of $d$, and let
$$P(d,k)\subset P(d)$$ be the set of partitions of $d$ into
at most $k$ parts. 
Let $|P(d,k)|$ be the
cardinality. To a partition{\footnote{The parts of
$\mathbf{p}$ are positive and satisfy $p_1\geq \ldots \geq p_\ell$.}}
$$\mathbf{p}= (p_1,\ldots,p_\ell) \in P(d,k),$$
we associate a $\kappa$ monomial by
$$\kappa_{\mathbf{p}} = \kappa_{p_1} \cdots \kappa_{p_\ell} \in 
\kappa^d(M_{g,n}^c) \ .$$

In \cite{kap1}, two basic facts about the $\kappa$ rings
of the moduli space of curves of compact type
are derived from Theorem \ref{mmmj}:
\begin{enumerate}

\item[$\bullet$]
the  canonical
 quotient,
$$\kappa^*(M_{0,2g+n}^c) \stackrel{\iota_{g,n}}{\rarr} 
\kappa^*(M^c_{g,n}) \rarr 0 \ $$
is an isomorphism for $n>0$,\vspace{5pt}
\item[$\bullet$]
a $\mathbb{Q}$-basis of $\kappa^d(M_{g,n}^c)$ is given by
$$\{ \kappa_{\mathbf{p}} \ | \ \mathbf{p} \in P(d,2g-2+n-d)\ \} \  $$
for $n>0$.
\end{enumerate}
The main tools used in \cite{kap1} are the virtual
geometry of the moduli space of stable quotients
\cite{MOP} and the intersection theory of strata
classes in the tautological ring $R^*(M_{g,n}^c)$.

By Theorem 
5 of \cite{kap1}, proven  unconditionally,
$$\text{dim}_{\mathbb{Q}} \ \kappa^d(M_{0,n}^c) = |P(d,n-2-d)|\ .$$
Hence, Theorem 1 is a consequence of the following
result.

\begin{Proposition} The space of 
relations among $\kappa$ monomials of degree $d$
valid in {\em all} the rings
$$\{ \kappa^*(M^c_{g,n}) \ |\  2g-2+n =\zeta \ \}$$ 
is of rank at least $|P(d)|-|P(d,\zeta-d)|$.
\end{Proposition}

Proposition 1 is proven in Sections 2 - 4 by constructing
universal   
relations in $\kappa^*(M_{g,n}^c)$
via the virtual geometry of the moduli space
of stable maps.
The interplay between stable quotients and stable maps
is an interesting aspect of the study of $\kappa^*(M^c_{g,n})$.

\subsection{Gorenstein conjecture}
The rank $g$ Hodge bundle over the moduli space of curves
$$\mathbb{E} \rarr \overline{M}_{g,n}$$
has fiber $H^0(C,\omega_C)$ over $[C,p_1,\ldots,p_n]$.
Let $$\lambda_k = c_k(\mathbb{E})$$
be the Chern classes.
Since $\lambda_g$ vanishes when restricted to 
$$\delta_0 = \overline{M}_{g,n} \setminus M_{g,n}^c\ ,$$
we obtain a well-defined
evaluation
$$\phi: A^*(M_{g,n}^c) \rarr \mathbb{Q}$$
given by integration
$$\phi(\gamma) = \int_{\overline{M}_{g,n}} \overline{\gamma} 
\cdot \lambda_g\ ,$$
where $\overline{\gamma}$ is any lift of $\gamma\in A^*(M_{g,n}^c)$ to
$A^*(\overline{M}_{g,n})$.

The tautological rings $R^*(M^c_{g,n})\subset A^*(M_{g,n}^c)$ have been
conjectured in \cite{FP1,Pand} to be Gorenstein algebras with socle in degree
$2g-3+n$,
$$\phi: R^{2g-3+n}(M_{g,n}^c) \stackrel{\sim}{\rarr} \mathbb{Q}\ .$$
As a consequence of Theorem 1 and the intersection calculations of
\cite{kap1}, we obtain the following result.

\begin{Theorem} \label{nmmg}
If $n>0$ and $\xi \in \kappa^d(M_{g,n}^c) \neq 0$, the linear
function 
$$L_\xi: R^{2g-3+n-d}(M_{g,n}^c)\rarr \mathbb{Q}$$
defined by the socle evaluation
$$L_\xi(\gamma) =\phi(\gamma\cdot \xi)$$
is non-trivial.
\end{Theorem}

Theorem \ref{nmmg}, 
discussed in Section \ref{lss},  may be viewed as significant 
evidence for the Gorenstein conjecture for
all $M_{g,n}^c$ with $n>0$.




\subsection{Acknowledgments}
The results here on the $\kappa$ rings 
were motivated by the study of stable quotients
 in
\cite{MOP}. Discussions with A. Marian and D. Oprea 
were very helpful.   
The methods developed  with C. Faber in \cite{FP}
played an important role.

The author was partially supported by NSF grant
DMS-0500187
 and the Clay institute.
 The research reported here was undertaken
while the author was visiting  MSRI in Berkeley 
and the 
Instituto
Superior T\'ecnico in Lisbon in the spring of
2009.

\section{$\kappa$ and $\psi$}
\subsection{$\psi$ classes}
Consider the cotangent line classes
$${\psi}_{n+1}, \ldots, {\psi}_{n+\ell} \in A^1(M_{g,n+\ell}^c)$$
at the last $\ell$ marked points.
Let
$$\epsilon^c: M_{g,n+\ell}^c \rarr M_{g,n}^c $$
be the proper forgetful map.
For each partition $\mathbf{p} \in P(d)$ of length $\ell$, we 
associate the class
\begin{equation*}
 \epsilon_*^c\left( \psi_{n+1}^{1+p_1} 
\cdots \psi_{n+\ell}^{1+p_\ell} \right)
\in A^d(M_{g,n}^c)\ .
\end{equation*}

The relation between the above  push-forwards of $\psi$ monomials and
the $\kappa$ classes is easily obtained.
For $\mathbf{p}=(d)$, we have
$$\epsilon_*^c(\psi_{n+1}^{1+d}) = \kappa_d$$
by definition.
The standard cotangent line comparison formulas
yield the length 2 case,
$$\epsilon_*^c(\psi_{n+1}^{1+p_1} \psi_{n+2}^{1+p_2}
) = \kappa_{p_1}\kappa_{p_2} + \kappa_{p_1+p_2}\ .$$
The full formula, due to Faber, is 
\begin{equation}
\label{vvrr}
 \epsilon_*^c\left( \psi_{n+1}^{1+p_1} 
\cdots \psi_{n+\ell}^{1+p_\ell} \right)
=
\sum_{\sigma\in S_\ell} \kappa_{\sigma(\mathbf{p})}\ ,
\end{equation}
where the sum is over the symmetric group $S_\ell$.
For $\sigma\in S_\ell$, let
$$\sigma = \gamma_1 \ldots \gamma_r$$
be the canonical cycle decomposition (including the 1-cycles), and
let $\sigma(\mathbf{p})_i$ be the sum of the parts of $\mathbf{p}$
with indices in the cycle $\gamma_i$.
Then,
$$\kappa_{\sigma(\mathbf{p})} = \kappa_{\sigma(\mathbf{p})_1}
\cdots \kappa_{\sigma(\mathbf{p})_r} \ .$$
A discussion of \eqref{vvrr} can be found in \cite{AC}.

\begin{Lemma} The sets of classes in $A^d(M_{g,n}^c)$ defined
by
$$ 
\{ \ \epsilon_*^c\left( \psi_{n+1}^{1+p_1} 
\cdots \psi_{n+\ell}^{1+p_\ell} \right) \ | \ \mathbf{p} \in P(d) \ \}
\ \ \ \text{and} \ \ 
\
\{ \ \kappa_{\mathbf{p}} \ |\ \mathbf{p} \in P(d) \ \}
$$ 
are related by an  invertible linear transformation independent
of $g$ and $n$.
\end{Lemma}

\begin{proof}
Formula \eqref{vvrr} defines a universal transformation
independent of $g$ and $n$. Since the transformation is
triangular in the partial ordering of $P(d)$
by length (with 1's on the diagonal), the invertibility
is clear.
\end{proof}

\subsection{Bracket classes}
Let $\mathbf{p} \in P(d)$ be a partition of length $\ell$. Let
\begin{equation}\label{brrg}
\langle \mathbf{p} \rangle = \epsilon_*^c \ \Big[
\prod_{i=1}^\ell \frac{1}{1-p_i \psi_{n+i}} \Big]^{\ell+d} \ \in
A^d(M_{g,n}^c)\ . 
\end{equation}
The superscript in the inhomogeneous 
expression
$\Big[
\prod_{i=1}^\ell \frac{1}{1-p_i \psi_{n+i}} \Big]^{\ell+d}$
indicates the summand in
$A^{\ell+d}(M_{g,n+\ell}^c)$.

We can easily expand definition \eqref{brrg}
to express the class $\langle \mathbf{p} \rangle$
linearly in terms of the classes 
$$\{ \ \epsilon_*^c\left( \psi_{n+1}^{1+p_1} 
\cdots \psi_{n+\ell}^{1+p_\ell} \right) \ | \ \mathbf{p} \in P(d) \ \} \ .$$
Since the string and dilation equation must be used
to remove the $\psi_{n+i}^0$ and $\psi_{n+i}^1$
factors, the transformation depends upon $g$ and $n$
only through $2g-2+n$.

\begin{Lemma} The sets of classes in $A^d(M_{g,n}^c)$ defined
by
$$\{ \  \langle \mathbf{p}\rangle \ |\ \mathbf{p} \in P(d) \ \} \ \ \ \text{and} \ \ 
\
\{ \ \epsilon_*^c\left( \psi_{n+1}^{1+p_1} 
\cdots \psi_{n+\ell}^{1+p_\ell} \right) \ | \ \mathbf{p} \in P(d) \ \}$$ 
are related by an  invertible linear transformation depending
only upon $2g-2+n$.
\end{Lemma}

\begin{proof}
Only the invertibility remains to be established.
The result exactly follows from the proof of Proposition 3 in 
\cite{FP}.
\end{proof}

By Lemma 1 and 2, the bracket classes lie in the
$\kappa$ ring,
$$\langle \mathbf{p} \rangle \in \kappa^d(M_{g,n}^c)\ .$$
We will prove Proposition 1 in the following 
equivalent form.

\begin{Proposition} The space of 
relations among the classes
$$\{ \  \langle \mathbf{p}\rangle \ |\ \mathbf{p} \in P(d) \ \}$$
valid in {\em all} the rings
$$\{ \kappa^*(M^c_{g,n}) \ |\  2g-2+n =\zeta \ \}$$ 
is of rank at least $|P(d)|-|P(d,\zeta-d)|$.
\end{Proposition}

\section{Relations via stable maps}
\subsection{Moduli of stable maps}
Let $\overline{M}_{g,n+\nn}(\proj^1,d)$ denote the moduli of stable
maps{\footnote{Stable maps were defined in \cite{Kont}, see \cite{FulP}
for an introduction.}} to $\proj^1$ of degree $d$, and let
$$\nu: \overline{M}_{g,n+\nn}(\proj^1,d)\rarr \overline{M}_{g,n}$$
be the morphism forgetting the map and the last $\nn$ markings.
The moduli space  
$$M_{g,n+\nn}^c(\proj^1,d) \subset \overline{M}_{g,n+\nn}(\proj^1,d)$$
is defined by requiring the domain curve to be of compact type.
The restriction
$$\nu^c: M_{g,n+\nn}^c(\proj^1,d) \rarr M_{g,n}^c$$
is proper and equivariant with respect to the symmetries
of $\proj^1$.

We will find relations in $A^*(M_{g,n}^c)$ by localizing 
$\nu^c$ push-forwards which vanish geometrically.
A complete analysis in the socle  $A^{2g-3}(M_g^c)$ was carried out
in \cite{FP}, but much more will be required for Theorem 1.
While the relations  in $A^*(M_{g,n}^c)$ via stable quotients \cite{kap1}
are more elegantly expressed, the ranks of the relations via stable
maps appear easier to compute.

\subsection{Relations}
\subsubsection{Indexing} \label{indd}
Let $d\leq 2g-3+n$, and let
$$\delta = 2g-3+n -d\ .$$
We will construct
a series of relations $I(g,d,\alpha)$  in $A^d(M_{g,n}^c)$
where
 $$\alpha=(\alpha_1, \ldots, \alpha_\nn)$$
is a (non-empty)
vector of non-negative integers satisfying two conditions:
\begin{enumerate}
\item[(i)] $|\alpha|= \sum_{i=1}^\nn \alpha_i \leq d-2-\delta$,
\item[(ii)] $\alpha_i>0$ for $i>1$.
\end{enumerate}
By condition (i), 
$d-2-\delta \geq 0$ so
$$d > g-1 +\lfloor \frac{n}{2} \rfloor\ .$$
Condition (ii) implies $\alpha_1$ is the only
integer permitted to vanish. 
The relation $I(g,d,\alpha)$  will be a variant
of the equations considered in \cite{FP}.

\subsubsection{Formulas}
Let
$\Gamma$ denote
the data type
\begin{equation}
\label{alll2}
(p_1,\ldots, p_\nn) \scup \{ p_{\nn+1}, \ldots, p_\ell\},
\end{equation}
satisfying
$$p_i >0,  \ \ \ \sum_{i=1}^\ell
p_i =d.$$
The first part of $\Gamma$ is an ordered $\nn$-tuple $(p_1,\ldots, p_\nn)$.
The second part
$\{ p_{\nn+1}, \ldots, p_\ell\}$
is an unordered set.
Let $\text{Aut}( 
\{ p_{\nn+1}, \ldots, p_\ell\})$ be the
group which permutes equal parts.
The group of 
automorphisms $\text{Aut}(\Gamma)$  
equals $\text{Aut}(\{ p_{\nn+1}, \ldots, p_\ell\})$.

\begin{Theorem} \label{gvv}
For all $\alpha$ satisfying (i-ii),
\begin{multline*}
\sum_{\Gamma} \frac{1}{|\text{\em Aut}(\Gamma)|} \ 
\prod_{i=1}^\nn p_i^{-\alpha_i} \prod_{i=\nn+1}^{\ell} (-p_i)^{-1}
\prod_{j=1}^\ell \frac{p_i^{p_i}}{p_i!} \ \ \lan p_1,\ldots,p_\ell \ran
\\ 
= 0\ 
\in A^{d}(M_{g,n}^c)
,
\end{multline*}
where the sum is over all  $\Gamma$ of type (\ref{alll2}).
\end{Theorem}

The bracket $\lan p_1,\ldots,p_\ell \ran \in A^d(M_{g,n}^c)$ denotes the
class associated to the partition defined by the union of
all the parts
$p_i$ of $\Gamma$.

\subsection{Proof of Theorem \ref{gvv}}
\label{pof3}

\subsubsection{Torus actions} 
\label{tactt}
The first step is to define the appropriate
torus actions.
Let $$\proj^1=\proj(V)$$ 
where $V=\com \oplus \com$.
Let $\com^*$ act diagonally on $V$:
\begin{equation}
\label{repp}
\xi\cdot (v_1,v_2) = ( v_1, 
\xi \cdot v_2).
\end{equation}
Let $\mathsf{p}_1, \mathsf{p}_2$ be the 
fixed points $[1,0], [0,1]$ of the corresponding
action on $\proj(V)$.
An equivariant lifting  of $\com^*$ to a line bundle $L$
over 
$\proj(V)$ is uniquely determined by the weights $[l_1,l_2]$
of the fiber
representations at the fixed points 
$$L_1= L|_{\mathsf{p}_1}, \ \ \ L_2= L|_{\mathsf{p}_2}.$$
The canonical lifting of $\com^*$ to the
tangent bundle $T_\proj$ has weights $[1,-1]$.
We will utilize the equivariant liftings of
$\com^*$ to $\oh_{\proj(V)}(1)$ and $\oh_{\proj(V)}(-1)$ with weights
$[1,0]$, $[0,1]$ respectively.

Over the moduli space of stable maps
 $\overline{M}_{g,n+\nn}(\proj(V), d)$, we have
$$\pi: U \rarr \overline{M}_{g,n+\nn}(\proj(V),d), 
\ \ \ \mu: U \rarr \proj(V)$$
where $U$ is the universal curve and $\mu$ is
the universal map.
The representation (\ref{repp}) canonically
induces $\com^*$-actions on $U$ and 
$\overline{M}_{g,n+\nn}(\proj(V),d)$ compatible
with the maps $\pi$ and $\mu$.
The $\com^*$-equivariant virtual class
$$[\overline{M}_{g,n+\nn}(\proj(V),d)]^{vir} \in A_{2g+2d-2+n+\nn}^{\com^*}
(\overline{M}_{g,n+\nn}(\proj(V),d))$$
will play an important role.

\subsubsection{Equivariant classes}
\label{ecc}
Three types of equivariant Chow classes on  
$\overline{M}_{g,n+\nn}(\proj(V),d)$ 
will be considered here:
\begin{enumerate}
\item[$\bullet$]
The linearization $[0,1]$ on $\oh_{\proj(V)}(-1)$
defines an $\com^*$-action on  
the rank $d+g-1$ bundle 
\begin{equation*}
\label{wqwq}
\mathbb{R}=R^1\pi_* (\mu^* \oh_{\proj(V)}(-1))
\end{equation*}
 on $\overline{M}_{g,n+\nn}(\proj(V),d)$.
Let $$c_{top}(\mathbb{R}) \in
A^{g+d-1}_{\com^*}(\overline{M}_{g,n+\nn}(\proj(V),d))$$
be the top Chern class. 


\item[$\bullet$]
For each marking $i$, let $\psi_i \in A^1_{\com^*}(\overline{M}_{g,n+\nn}
(\proj(V),d)$
denote the first Chern class of 
the canonically linearized cotangent line corresponding to
$i$.

\item[$\bullet$] Denote the   $i^{th}$ evaluation morphism
by
$$\text{ev}_i: \overline{M}_{g,n+\nn}(\proj(V),d) \rarr \proj(V).$$
With $\com^*$-linearization $[1,0]$ on
$\oh_{\proj(V)}(1)$,
let
$$\rho_i =c_1( \text{ev}_i^* \oh_{\proj(V)}(1)) \in A^1_{\com^*}
(\overline{M}_{g,n+\nn}(\proj(V),d)\ .$$
With $\com^*$-linearization $[0,-1]$ on
$\oh_{\proj(V)}(1)$,
let
$$\widetilde{\rho}_i =c_1( \text{ev}_i^* \oh_{\proj(V)}(1)) \in A^1_{\com^*}
(\overline{M}_{g,n+\nn}(\proj(V),d)\ .$$
\end{enumerate}
In the non-equivariant limit, $\rho_i^2=0$.
Our notation here closely follows \cite{FP}.

\subsubsection{Vanishing integrals}
\label{vanny}
The forgetful morphism
$$\nu: \overline{M}_{g,n+\nn}(\proj(V),d) \rarr \overline{M}_{g,n}$$
is $\com^*$-equivariant with respect to the trivial action
on $\overline{M}_{g,n}$.
As in Section \ref{indd},
let $$d\leq 2g-3+n, \ \ \ 
\delta = 2g-3+n -d,$$
and let $\alpha=(\alpha_1, \ldots, \alpha_\nn)$
satisfy
\begin{enumerate}
\item[(i)] $|\alpha|= \sum_{i=1}^\nn \alpha_i \leq d-2-\delta$,
\item[(ii)] $\alpha_i>0$ for $i>1$.
\end{enumerate}
Let $I(g,d,\alpha)$ be the $\com^*$-equivariant push-forward
\begin{multline*}
 \nu_*\left(
\rho_{n+1}^{d-1-\delta-|\alpha|} \ 
\prod_{i=1}^\nn \rho_{n+i} \psi_{n+i}^{\alpha_i} \
\prod_{{j}=1}^n \widetilde{\rho}_{j} 
\ c_{top}(\mathbb{R})  
\ \cap \
[\overline{M}_{g,n+\nn}(\proj(V),d)]^{vir}
\right)\ .
\end{multline*}
The degree of the class 
$$\rho_{n+1}^{d-1-\delta-|\alpha|} \ 
\prod_{i=1}^\nn \rho_{n+i} \psi_{n+i}^{\alpha_i}\
\prod_{{j}=1}^n \widetilde{\rho}_{j}
\ c_{top}(\mathbb{R})$$
is easily computed to be
\begin{multline*}
 d-1-\delta -|\alpha| + \nn +|\alpha|+n +d+g-1  =
\\
g+2d-2+n+\nn -\delta\ .
\end{multline*}
Since the cycle dimension of the virtual class is
$2g+2d-2+n+\nn$,
the push-forward $I(g,d,\alpha)$ has cycle dimension
\begin{eqnarray*}
2g+2d-2+n+m - (g+2d-2 +n+\nn -\delta) &= &g+\delta\\
& = &
 3g-3+n -d\ .
\end{eqnarray*} 
Equivalently,
$I(g,d,\alpha)\in A^d_{\com^*}(\overline{M}_{g,n})$.
Since the class $\rho_{n+1}$ appears
with exponent $$d-\delta -|\alpha|\geq 2,$$
$I(g,d,\alpha)$ vanishes in the non-equivariant limit.

\subsubsection{Localization terms}
\label{lres}
The virtual localization formula of \cite{GP} calculates
$I(g,d,\alpha)$ 
in terms of tautological classes  on the moduli space
$\overline{M}_{g,n}$.
To prove Theorem \ref{gvv}, we will calculate the
restriction of the localization formula to $M_{g,n}^c$.

The localization formula 
expresses $I(g,d,\alpha)$ as a  sum over connected decorated 
graphs $\Gamma$ indexing the $\com^*$-fixed loci of 
$\overline{M}_{g,n+\nn}(\proj(V),d)$.
The vertices of the graphs lie over the
fixed points $\mathsf{p}_1, \mathsf{p}_2 \in \proj(V)$ and are
labelled with genera (which sum over the graph to $g-h^1(\Gamma)$).
The edges of the graphs lie over $\proj^1$ and
are labelled with degrees (which sum over the 
graph to $d$). Finally, the graphs carry
$n+\nn$ markings on the vertices. 
The valence $\text{val}(v)$ of a vertex $v\in \Gamma$
counts both the
incident edges and markings.
The edge valence of $v$
counts only the incident edges.

Only a very restricted subset of graphs 
will yield non-vanishing contributions to $I(g,d,\alpha)$
in the non-equivariant limit.
If a graph $\Gamma$
contains a vertex lying over $\mathsf{p}_1$ of 
edge valence greater than 1, then the contribution of
$\Gamma$ to vanishes
by our choice of linearization
on the bundle $\mathbb{R}$. 
A vertex over $\mathsf{p}_1$ of edge valence greater than 1
yields a trivial Chern root of $\mathbb{R}$ (with trivial weight 0) in the
numerator of the localization formula to force the vanishing.

By the above vanishing, only {\em comb} graphs $\Gamma$ 
contribute to $I(g,d,\alpha)$. Comb graphs contain
$\ell\leq d$ vertices lying over $\mathsf{p}_1$ each
connected by a distinct edge to a unique vertex lying
over $\mathsf{p}_2$.

If $\Gamma$ contains
a vertex over $\mathsf{p}_1$ of positive genus, then the
restriction to $M_{g,n}^c$ of the
contribution of $\Gamma$ to $I(g,d,\alpha)$ vanishes
by the following argument.
Let $v$ be a genus $g(v)>0$
vertex lying over $\mathsf{p}_1$. The integrand term $c_{top}(\mathbb{R})$
yields a factor $c_{g(v)}(\hodge^*)$ with trivial $\com^*$-weight on the
genus $g(v)$ moduli space corresponding to the vertex $v$.
Since
$$\lambda_{g(v)} | _{M_{g(v),\text{val}(v)}^c} = 0$$
by \cite{VdG}, the required vanishing holds.

The linearizations of the classes 
$\rho_i$ and
$\widetilde{\rho}_{j}$  
place restrictions on
the marking distribution.
Since the class $\widetilde{\rho}_{j}$ is
obtained from $\oh_{\proj(V)}(1)$ with linearization
$[0,-1]$, the first $n$  markings must lie on the
unique vertex over  over $\mathsf{p}_2$.
Since the class $\rho_i$ is
obtained from $\oh_{\proj(V)}(1)$ with linearization
$[1,0]$, the last $\nn$ markings must lie on vertices over $\mathsf{p}_1$.

Finally, we claim the last $\nn$ 
markings of $\Gamma$ must lie on {distinct}
vertices over $\mathsf{p}_1$ for nonvanishing contribution to
$I(g,d,\alpha)$.
Let $v$ be a vertex over $\mathsf{p}_1$ (with $g(v)=0$).
If $v$ carries at least two markings, the 
fixed locus corresponding to $\Gamma$ 
contains a product factor $\overline{M}_{0,r+1}$
where $r$ is the number of markings incident to $v$.
The classes $\psi_{n+i}^{\alpha_i}$
 carry trivial
$\com^*$-weight. Moreover, as each
$\alpha_i > 0$ for $i>1$, we see the sum of the
$\alpha_i$ as $i$ ranges over the set of markings
incident to $v$  is at least $r-1$. Since
the sum exceeds the dimension of $\overline{M}_{0,r+1}$,
the graph contribution to  $I(g,d,\alpha)$ vanishes.

The proof of  the main result about the 
localization terms for $I(g,d,\alpha)$ is
now complete.

\begin{Proposition}
\label{qqqq}
The restriction of 
$I(g,d,\alpha)$ to $M_{g,n}^c$ is
 expressed via the virtual localization formula as
a sum over 
genus $g$, degree $d$, marked comb graphs $\Gamma$ satisfying:
\begin{enumerate}
\item[(i)] all vertices over $\mathsf{p}_1$ are of genus 0,

\item[(ii)] the unique vertex over $\mathsf{p}_2$ carries all of the first $n$
markings,
\item[(iii)] the last $m$ markings all lie over $\mathsf{p}_1$,
\item[(iv)] 
each vertex over $\mathsf{p}_1$ carries at most 1 of the last $m$ markings.

\end{enumerate}
\end{Proposition}

\subsubsection{Formulas}
\label{exx}
The precise contributions of allowable graphs
$\Gamma$ to the non-equivariant limit of $I(g,d,\alpha)$ 
are now calculated.

Let $\Gamma$ be a
genus $g$, degree $d$, comb graph  with $n+\nn$ 
markings satisfying
conditions (i-iv) of Proposition \ref{qqqq}.
By condition (iv), $\Gamma$ must have  $\ell\geq \nn$ edges.
$\Gamma$ may be described uniquely
by the data 
\begin{equation}
\label{alll}
(p_1,\ldots, p_\nn) \scup \{ p_{\nn+1}, \ldots, p_\ell\},
\end{equation}
satisfying:
$$p_i >0,  \ \ \ \sum_{i=1}^\ell
p_i =d.$$
The elements of the ordered $\nn$-tuple $(p_1,\ldots, p_\nn)$
correspond
to the degree assignments of the edges incident
to the vertices marked by the last $\nn$ markings. 
The elements of the unordered partition
$\{ p_{\nn+1}, \ldots, p_\ell\}$
correspond to the degrees of edges incident to
the unmarked vertices over $\mathsf{p}_1$.
The group of graph 
automorphisms  is
$$\text{Aut}(\Gamma)= \text{Aut}(\{ p_{\nn+1}, \ldots, p_\ell\})\ .$$

By a direct application of the virtual localization
formula of \cite{GP}, we find the contribution of the graph
(\ref{alll})
to the normalized{\footnote{The
parallel equation on page 106 of \cite{FP}
has a sign error in the normalization.
Instead of $(-1)^{g+1} I(g,d,\alpha)$ there,
the normalization should be $(-1)^{g+1+|\alpha|+\ell(\alpha)}I(g,d,\alpha)$. 
The sign change makes no difference.}}
push-forward
$$(-1)^{g+1+|\alpha|+n+\nn}\cdot I(g,d, \alpha)$$
equals
\begin{equation*}
\frac{1}{|\text{Aut}(\Gamma)|} \ 
\prod_{i=1}^\nn p_i^{-\alpha_i} \prod_{i=\nn+1}^{\ell} (-p_i)^{-1}
\prod_{i=1}^\ell \frac{p_i^{p_i}}{p_i!} \ \ \lan p_1,\ldots,p_\ell\ran\ .
\end{equation*}
Hence, the vanishing of $I(g,d,\alpha)$ yields the relation
\begin{equation*}
\sum_{\Gamma} \frac{1}{|\text{Aut}(\Gamma)|} \ 
\prod_{i=1}^\nn p_i^{-\alpha_i} \prod_{i=\nn+1}^{\ell} (-p_i)^{-1}
\prod_{i=1}^\ell \frac{p_i^{p_i}}{p_i!} \ \ \lan p_1,\ldots,p_\ell\ran= 0\ ,
\end{equation*}
where the sum is over all  graphs (\ref{alll}). \qed

\begin{Question}
Are the relations of Theorem \ref{gvv} equivalent to
relations constructed in Section 3 of \cite{MOP}
via stable quotients? 
\end{Question}

\section{Rank analysis}
\subsection{Matrix of relations}
Theorem \ref{gvv} yields relations in $\kappa^d(M_{g,n}^c)$,
indexed by $\alpha=(\alpha_1, \ldots,\alpha_\nn)$ satisfying conditions (i-ii)
of Section \ref{indd} with
$$\delta = 2g-3+n-d \geq 0 . $$
We rewrite the relation obtained from the vanishing 
of $I(g,d,\alpha)$ as
\begin{equation} \label{syss}
\sum_{\mathbf{p}\in P(d)} \mathsf{C}_{\alpha}^{\mathbf{p}} \ \langle
\mathbf{p} \rangle = 0 \ .
\end{equation}
The coefficients are
$$\mathsf{C}_{\alpha}^{\mathbf{p}}= 
\frac{1}{|\text{Aut}(\mathbf{p})|} \prod_{i=1}^\ell\frac{p_i^{p_i}}{p_i!}
\ \sum_{\phi}\  \prod_{i=1}^\nn p_{\phi(i)}^{-\alpha_i} 
 \prod_{j\in \text{Im}(\phi)^c} (-p_j)^{-1}\ ,$$
where the sum is over all injections 
$$\phi: \{ 1, \ldots,\nn \} \rarr \{ 1, \ldots, \ell \}  $$
and 
$$\text{Im}(\phi)^c \subset \{1, \ldots, \ell\}$$
is the complement of the image of $\phi$.

To prove Proposition 2, 
we will show the system \eqref{syss}  
is of rank
at least $|P(d)| - |P(d,\delta+1)|$.
The claim is empty unless
$0 \leq \delta \leq d-2$.

\subsection{Ordering}
For $0\leq \delta \leq d-2$,
define the subset $P_\delta(d) \subset P(d)$
by removing partitions of length at most $\delta+1$,
$$P_\delta(d) = P(d) \setminus P(d,\delta+1)\ .$$
We order $P_\delta(d)$ by the following rules
\begin{enumerate}
\item[$\bullet$] longer partitions appear before shorter partitions,
\item[$\bullet$] for partitions of the same length, we use the
lexicographic ordering with
larger parts{\footnote{Remember the parts of
$\mathbf{p}=(p_1,\ldots, p_\ell)$ are
ordered by $p_1\geq \ldots \geq p_\ell$.}}
appearing before smaller parts.
\end{enumerate}
For example, 
the ordered list of  the 10 elements of $P_0(6)$ is
$$(1^6), \ (2,1^4), \ (3,1^3), \ (2^2,1^2), \ (4,1^2), \ (3,2,1), \ (2^3),
\ (5,1), \ (4,2), \ (3,3) \ .$$

Given a partition $\mathbf{p}\in P{(d)}$, let $\widehat{\mathbf{p}}$
be the partition obtained removing all parts equal to $1$.
For example,
$$\widehat{(1^6)} = \emptyset, \ \ \widehat{(3,2,1)} = (3,2)\ .$$
Let $\mathbf{p}^{-}$ be the partition obtained by lowering
all the parts of $\mathbf{p}$ by 1,
$$(1^6)^- = \emptyset, \ \  (3,2,1)^- = (2,1)\ .$$
If $\mathbf{p}$ has length $\ell$, then 
$$ \mathbf{p}^- \in P(d-\ell). \ $$

To each partition $\mathbf{p} \in P_\delta(d)$,
we associate data
$\alpha[\mathbf{p}]$
satisfying conditions (i)-(ii) with respect to $\delta$
by the following rules.
The special designation
$$\alpha[(1^d)] = (0) $$
is given.
Otherwise
$$\alpha[\mathbf{p}]
 = \mathbf{p}^-\ . $$
We note condition (i) of Section \ref{indd},
$$|\alpha[\mathbf{p}]| \leq d -2- \delta\ , $$
is satisfied in all cases.

Let $\mathsf{M}_\delta(d)$ be the square matrix indexed
by the
ordered set $P_\delta(d)$ with
elements
$$M_\delta(d)[ \mathbf{p}, \mathbf{q}] = 
\mathsf{C}_{\alpha[\mathbf{p}]}^{\mathbf{q}} \       .$$
The rank of the system \eqref{syss} is at least
$$|P_\delta(d)| = |P(d)| - |P(d,\delta+1)|$$
by the following nonsingularity result proven in 
Sections \ref{scc} - \ref{kjh} below.

\begin{Proposition} \label{nlll}
For $0\leq \delta \leq d-2$, the matrix
$\mathsf{M}_\delta(d)$ is nonsingular.
\end{Proposition}

Proposition \ref{nlll} implies Proposition 2 and
thus Theorem 1. Moreover, Proposition \ref{nlll}
provides a new approach to \cite{FP}.

\subsection{Scaling} \label{scc}
Let $\mathsf{X}_\delta(d)$ be the square matrix
 indexed by the
ordered set $P_\delta(d)$ with
elements
\begin{eqnarray*}
\mathsf{X}_\delta(d)[ (1)^d, \mathbf{q}] & = & 
(-1)^{\ell(\mathbf{q})-1} d
\\
\mathsf{X}_\delta(d)[ \mathbf{p}\neq (1)^d, \mathbf{q}] & = & 
 \sum_{\phi}\  
(-1)^{\ell(\mathbf{q})- \ell(\widehat{\mathbf{p}})}
\prod_{i=1}^{\ell(\widehat{\mathbf{p}})} q_{\phi(i)}^{-\widehat{p}_i+2} \ ,
\end{eqnarray*}
where the sum is over all injections 
$$\phi: \{ 1, \ldots,\ell(\widehat{\mathbf{p}}) \} \rarr 
\{ 1, \ldots, \ell(\mathbf{q}) \}\ .  $$
For example, $\mathsf{X}_{0}(6)$ is
\vspace{10pt}
$$
\left( \begin{array}{rrrrrrrrrr}

 -6 &    6  &    -6 &   -6 &   6 &   6 &   6 &   -6 &   -6 &  -6 \vspace{4pt}\\
 -6 &   5 &   -4 &   -4 &   3 &   3 &   3 &   -2 &   -2 &   -2  \vspace{4pt} \\
 -6 & \frac{9}{2} & -\frac{10}{3} & -3 & \frac{9}{4} & 
\frac{11}{6} & \frac{3}{2} & -\frac{6}{5} & -\frac{3}{4} & -\frac{2}{3} 
\vspace{4pt}
\\
 30 &   -20 &   12 &   12 &   -6 &   -6 &   -6 &   2 &  2 &  2 
\vspace{4pt}
\\
       -6 & \frac{17}{4} &  -\frac{28}{9} &  -\frac{5}{2} &
 \frac{33}{16} & \frac{49}{36}& \frac{3}{4} &- \frac{26}{25}
 & -\frac{5}{16} & -\frac{2}{9} \vspace{4pt}\\
30 &  -18 &  10 &  9 &  -\frac{9}{2} &  -\frac{11}{3} &  -3 
& \frac{6}{5} & \frac{3}{4} & \frac{2}{3} \vspace{4pt}\vspace{4pt}\\
-120 &   60 &   -24 &   -24 &   6 &   6 &   6 &   0 &  0 &  0 \vspace{4pt}\\
-6 & \frac{33}{8} & - \frac{82}{27} & -\frac{9}{4} &
\frac{129}{64} & \frac{251}{216} & \frac{3}{8} & 
-\frac{126}{125} & -\frac{9}{64} & -\frac{2}{27}\vspace{4pt} \\
30 & -17 & \frac{28}{3} & \frac{15}{2} & -\frac{33}{8} &
-\frac{49}{18} & -\frac{3}{2} & \frac{26}{25} & \frac{5}{16} &
\frac{2}{9} \vspace{4pt} \\
30 &  -16 &  8 &  \frac{13}{2} &  -3 &  -2 &  -\frac{3}{2} &
 \frac{2}{5} & \frac{1}{4} & \frac{2}{9} 
\end{array}
\right) \ .
$$
\vspace{10pt}

The matrix 
$\mathsf{X}_\delta(d)$
is obtained from $\mathsf{M}_\delta(d)$ 
by dividing each
column corresponding to $\mathbf{q}$ by 
$$\frac{1}{|\text{Aut}(\mathbf{q})|} \prod_{i=1}^{\ell(\mathbf{q})}
\frac{q_i^{q_i-1}}{q_i!}\ .$$
Hence, 
$\mathsf{X}_\delta(d)$ is nonsingular if and only if
$\mathsf{M}_\delta(d)$ is nonsingular.

\subsection{Elimination} \label{ellm}
Our strategy for proving 
Proposition \ref{nlll} is to find an upper-triangular square
matrix $\mathsf{Y}_0(d)$ for which the product
\begin{equation}\label{hh5}
\mathsf{X}_0(d) \cdot \mathsf{Y}_0(d)
\end{equation}
is lower-triangular with $\pm1$'s on the diagonal.
Since $\mathsf{X}_\delta(d)$ for
$$0 \leq \delta \leq d-2$$
occurs
as an upper left minor of $\mathsf{X}_0(d)$, the
lower-triangularity of the product 
\eqref{hh5} will establish Proposition \ref{nlll} for the
full range of $\delta$ values.

We define $\mathsf{Y}_0(d)$ to be the square matrix
 indexed by the
ordered set $P_0(d)$ given by the following rules.
The upper left corner is
$$
\mathsf{Y}_0(d)[ (1^d),(1^d) ] = \frac{1}{d} 
$$
If at least one of $\{\mathbf{p}, \mathbf{q}\}$ is not equal to
$(1^d)$, then the matrix elements are
\begin{multline*}
\mathsf{Y}_0(d)[ {\mathbf{p}}, {\mathbf{q}}]  = \\
\frac{1}{|\text{Aut}({\mathbf{p}})|} 
\frac{1}{|\text{Aut}({\mathbf{\widehat{q}}})|}
 \sum_{\theta}\  
\prod_{i=1}^{\ell({\mathbf{q}})} 
\binom{{q}_i}{{p}_{i[1]}, \ldots, {p}_{i[\ell_i]}}
{q}_i^{ \ell_i-2}
\prod_{j=1}^{\ell_i} {p}_{ij}^{{p}_{ij}-1}
 \ ,
\end{multline*}
where the sum is over all functions
$$\theta: \{ 1, \ldots,\ell({\mathbf{p}}) \} \rarr 
\{ 1, \ldots, \ell(\mathbf{q}) \}   $$
with
$$\theta^{-1}(i) = \{ {i[1]}, \ldots, {i[\ell_i]} \}$$
satisfying
$${q}_i = \sum_{j=1}^{\ell_i} {p}_{i[j]}\  . $$
For example, $\mathsf{Y}_{0}(6)$ is
\vspace{10pt}
$$
\left( \begin{array}{rrrrrrrrrr}
\frac{1}{6} & 1 & 3 & \frac{1}{2}&  16 &    3 &  \frac{1}{6} &
    125 &    16 &  \frac{9}{2} \vspace{4pt} \\
  0 &    1 &   6 &     1 &    48 &    9 &    \frac{1}{2} &   500&
    64 &   18 \vspace{4pt}\\
  0 &    0 &  3 &     0 &  36&    3&     0 &  450&    36 &    9 \vspace{4pt}\\
     0 &    0 &    0 & \frac{1}{2} &   12&    6 &  \frac{1}{2} &
   300 &    60 &   18 \vspace{4pt}\\
  0 &    0 &    0 &    0 &    16&    0&     0 & 320&    16&     0 
\vspace{4pt} \\
  0 &    0 &    0 &     0 &      0&    3&     0 &   180&    36&    18 
\vspace{4pt}\\
   0 &    0 &    0 &    0  &     0&    0&  \frac{1}{6}&       0&    12&     0 
\vspace{4pt}\\
   0 &    0  &   0  &   0  &     0 &   0 &    0&      125&     0&     
0 \vspace{4pt}\\
   0 &    0  &   0  &   0 &      0 &   0 &    0 &        0 &   16  &   
0 \vspace{4pt}\\
    0 &    0 &    0  &   0 &      0&    0&     0&         0&     0&
\frac{9}{2}
\end{array} \right) .
$$

By the conditions on $\theta$ in the definition, 
$\mathsf{Y}_0(d)$ is easily seen to be upper-triangular.

\subsection{Generating functions}
Let $\Q[t]$ denote the
polynomial ring in infinitely many variables
$$t=\{t_1, t_2, t_3, \ldots\} \ . $$
Define a $\Q$-linear function
$$\lan \ \ran : \Q[t] \rarr \Q$$
by the equations $\lan 1 \ran=1$ and
$$\lan t_{d_1} t_{d_2} \cdots t_{d_k} \ran = 
(d_1+ d_2+\ldots+ d_k)^{k-3} \ .$$
We may extend $\lan \ \ran$ uniquely to define a $x$-linear function:
$$\lan \ \ran: \Q[t] [[x]] \rarr \Q[[x]].$$

For each non-negative integer $i$, let
$$Z_i(t,x) = \sum_{j>0} x^j t_j \frac{j^{j-i}}{j!} \in \Q[t][[x]].$$
Applying the bracket, we define
$$\mathsf{F}_{\alpha_1,\ldots,\alpha_m} =
\lan \exp(-Z_1) \cdot Z_{\alpha_1} \cdots Z_{\alpha_m} \ran\ \in
\mathbb{Q}[[x]] .$$

\begin{Lemma}\label{yh2}
Let
$\alpha=(\alpha_1, \ldots, \alpha_n)$
be a  non-empty sequence of non-negative
integers satisfying $\alpha_i>0$ for $i>1$. 
The series
$$\mathsf{F}_{\alpha_1,\ldots,\alpha_m} 
\in \Q[[x]]$$
is a {\em polynomial} of degree at most $1+\sum_{i=1}^m \alpha_i$
in $x$.
\end{Lemma}

\begin{Lemma} \label{yh3}
Let $\alpha_1\geq 0$. 
Then,
$$
\mathsf{F}_{\alpha_1} = 
 \frac{(-1)^{\alpha_1}}{(1+\alpha_1) (1+\alpha_1)!} x^{1+\alpha_1} + \ldots$$
where the dots stand for lower order terms.
\end{Lemma}

Lemma \ref{yh2} can be proven by various methods. A proof
via localization on moduli space is given in \cite{FP} in Section 1.7. \qed
\vspace{10pt}

Lemma \ref{yh3} is more interesting. 
The integral
\begin{equation}\label{hnkk}
J_{1+\alpha_1}=
\int_{\overline{M}_{0,1}(\proj^1,1+\alpha_1)} 
\rho_1 \psi_1^{\alpha_1} \ c_{top}(\mathbb{R})
\end{equation}
can be evaluated by exactly following{\footnote{The
equivariant lifts are taken just as in Section \ref{ecc}.}} 
the localization analysis of
Section \ref{pof3}.
We find 
\begin{equation*}
J_{1+\alpha_1}= (-1)^{\alpha_1}
\sum_{\Gamma}
\frac{1}{|\text{Aut}(\Gamma)|} \ 
 p_1^{-\alpha_1} \prod_{i=2}^{\ell} (-p_i)^{-1}
\prod_{i=1}^\ell \frac{p_i^{p_i}}{p_i!} \ \ (1+\alpha_1)^{\ell-3}\ 
\end{equation*}
where the sum is over all 1-pointed
comb graphs \eqref{alll} of total degree $1+\alpha_1$ .
We conclude  $J_{1+\alpha_1}$ equals, up to the
factor of $(-1)^{\alpha_1}$, the leading  $x^{1+\alpha_1}$
coefficient
of $\lan \exp(-Z_1) \cdot Z_{\alpha_1} \ran$.

To calculate the integral \eqref{hnkk}, we use well-known
equations in Gromov-Witten theory.
Certainly 
\begin{equation} \label{bnp}
J_1=1\ .
\end{equation}
By two applications of the divisor equation,
\begin{equation*}
k^2 J_k = \int_{\overline{M}_{0,3}(\proj^1,k)} 
\rho_1 \psi_1^{k-1} \rho_2 \rho_3 \ c_{top}(\mathbb{R})
\end{equation*}
By the topological recursion relation \cite{CK} applied to
the right side,
$$k^2 J_k =   \int_{\overline{M}_{0,2}(\proj^1,k-1)} 
\rho_1 \psi_1^{k-2} \rho_2  \ c_{top}(\mathbb{R})\ \cdot \
\int_{\overline{M}_{0,3}(\proj^1,1)} 
\rho_1  \rho_2 \rho_3 \  c_{top}(\mathbb{R}) \ .$$
We obtain the recursion
\begin{eqnarray*}
k^2 J_k & =& (k-1) J_{k-1} J_1 
\\ & = & (k-1) J_{k-1}
\end{eqnarray*}
which we can easily solve 
$$J_k = \frac{1}{k\cdot k!}\ $$
starting with the initial condition \eqref{bnp}.\qed

The case where the $\alpha$ data is empty will arise
naturally. We define
$$\mathsf{F}_{\emptyset} = \lan \exp(-Z_1) \ran .$$
The following result is derived from Lemma \ref{yh2} by the relation
$$x\frac{d}{dx} \mathsf{F}_{\emptyset} = - \mathsf{F}_0.$$

\begin{Lemma}\label{yh4}
$\mathsf{F}_{\emptyset} = 1-x .$
\end{Lemma}

\subsection{Product} \label{kjh}

We will now prove the basic identity
\begin{equation}\label{hh6}
\mathsf{X}_0(d) \cdot \mathsf{Y}_0(d) = \mathsf{L}_0(d)
\end{equation}
where $\mathsf{L}_0(d)$ is lower triangular with
diagonal entries all $\pm 1$.

We first address the special 
upper left corner. The product on the left
side of \eqref{hh6} is
$$
\mathsf{L}_0(d)[(1^d),(1^d)] =
(-1)^{d-1} d \cdot \frac{1}{d} = (-1)^{d-1}\ ,$$
a diagonal entry of the specified form.

Next assume $\mathbf{p}\neq (1^d)$.
Then, the matrix elements are
\begin{equation}\label{gttp}
\mathsf{L}_0(d)[ {\mathbf{p}}, {\mathbf{q}}]  = 
\frac{1}
{|\text{Aut}({\mathbf{\widehat{q}}})|}
 \sum_{\gamma}\  
\prod_{i=1}^{\ell({\mathbf{q}})} 
\text{Coeff}(F_{\gamma^{-1}(i)}, x^{q_i})\  q_i \ q_i!\ ,
\end{equation}
where the sum is over all functions
$$\gamma: \{ 1, \ldots,\ell({\mathbf{\widehat{p}}}) \} \rarr 
\{ 1, \ldots, \ell(\mathbf{q}) \}  \ . $$
In case
$\gamma^{-1}(i)= \{ {i[1]}, \ldots, {i[\ell_i]} \}$ is nonempty,
we define
$$\mathsf{F}_{\gamma^{-1}(i)} = \mathsf{F}_{\widehat{p}_{i[1]}-1,
\ldots, \widehat{p}_{i[\ell_i]}-1}\ .$$
If $\gamma^{-1}(i) = \emptyset$, then
$$\mathsf{F}_{\emptyset} = \lan \exp(-Z_1) \ran = 1 -x.$$
Equation \eqref{gttp} is obtained from a
simple unravelling of the definitions.

If $q_i>1$, $\text{Coeff}(F_{\gamma^{-1}(i)}, x^{q_i})$
vanishes unless
$\gamma^{-1}(i)$ is nonempty by Lemma \ref{yh4} and unless
\begin{equation}\label{bnyy}
q_i \leq 1 -\ell_i+ \sum_{j=1}^{\ell_i} \widehat{p}_{i[j]} \ 
\end{equation}
by Lemma \ref{yh2}.
Inequality \eqref{bnyy} for all parts $q_i>1$ implies
$$\ell(\mathbf{q}) \geq \ell(\mathbf{p}) \ .$$
Moreover, if equality of length holds, then inequality
\eqref{bnyy} implies either
$\mathbf{q}$ precedes $\mathbf{p}$
in the ordering of $P_0(d)$ or $\mathbf{q}= \mathbf{p}$.

We conclude the matrix 
$\mathsf{L}_0(d)$
is lower-triangular when the first coordinate
$\mathbf{p}$ is not $(1)^d$. The diagonal elements 
for $\mathbf{p}\neq (1^d)$
are
$$\mathsf{L}_0(d)[ {\mathbf{p}}, {\mathbf{p}}] = \prod_{i=1}^
{\ell(\widehat{\mathbf{p}})} (-1)^{\widehat{p}_i-1}
\cdot (-1)^{\ell({\mathbf{p}})- \ell(\widehat{\mathbf{p}})}$$
by Lemmas \ref{yh3} and \ref{yh4}.

To complete the proof of the lower-triangularity of $\mathsf{L}_0(d)$,
we must show the vanishing of 
$\mathsf{L}_0(d)[(1^d), \mathbf{q}\neq (1^d)]$.
The
matrix elements are
\begin{equation*}
\mathsf{L}_0(d)[ (1^d), {\mathbf{q}}\neq (1^d)]  = 
\frac{1}
{|\text{Aut}({\mathbf{\widehat{q}}})|}
 \sum_{\tilde{\gamma}}\  
\prod_{i=1}^{\ell({\mathbf{q}})} 
\text{Coeff}(\widetilde
{F}_{\tilde{\gamma}^{-1}(i)}, x^{q_i})\  q_i \ q_i!\ ,
\end{equation*}
where the sum is over all functions
$$\tilde{\gamma}: \{ 1 \} \rarr 
\{ 1, \ldots, \ell(\mathbf{q}) \}  \ . $$
In case
$\tilde{\gamma}^{-1}(i)= \{ 1 \}$ is nonempty,
we define
$$\widetilde{\mathsf{F}}_{\tilde{\gamma}^{-1}(i)} = \mathsf{F}_{0}\ .$$
If $\tilde{\gamma}^{-1}(i) = \emptyset$, then
$$\widetilde{\mathsf{F}}_{\emptyset} = \lan \exp(-Z_1) \ran = 1 -x.$$

Let $q_1>1$ be the largest part of $\mathbf{q}$. Then
$$\text{Coeff}(\widetilde{F}_{\tilde{\gamma}^{-1}(1)}, x^{q_1})=0$$
by Lemmas \ref{yh2} and \ref{yh4}. Hence,
$$\mathsf{L}_0(d)[(1^d), \mathbf{q}\neq (1^d)] =0,$$
and the lower-triangularity of $\mathsf{L}_0(d)$
is fully proven.

The proof of Proposition \ref{nlll} is complete.
Following the implications back, the proof of Theorem 1
is also complete. \qed

Since we know explicitly the diagonal elements of the
triangular matrices $\mathsf{Y}_0(d)$ and $\mathsf{L}_0(d)$,
the product 
$$\mathsf{X}_0(d) \cdot \mathsf{Y}_0(d) = \mathsf{L}_0(d)$$
yields a simple formula for the determinant,
$$\text{det}( \mathsf{X}_{0,d} ) = (-1)^{d-1}\prod_{\mathbf{p}\in 
P_0(d)\setminus\{(1^d)\}} \left(
\frac{|\text{Aut}(\widehat{\mathbf{p}})|}
{\prod_{i=1}^{\ell(\mathbf{p})} 
p_i^{p_i-2}}
\
(-1)^{\ell({\mathbf{p}})} 
\prod_{i=1}^
{\ell(\widehat{\mathbf{p}})} (-1)^{\widehat{p}_i} \right)
\ .$$

\section{Gorenstein conjecture}
\subsection{Proof of Theorem \ref{nmmg}}
\label{lss}
If $n>0$, the pairing 
$$\kappa^d(M_{g,n}^c) \times R^{2g-3+n-d}(M_{g,n}^c) \rarr \mathbb{Q}$$
is shown to have rank at least 
$|P(d, 2g-2+n-d)|$ in \cite{MOP}. Since 
$$\text{dim}_{\mathbb{Q}} \kappa^d(M_{g,n}^c) = |P(d, 2g-2+n-d)|$$
by Theorem 1 and \cite{MOP}, Theorem \ref{nmmg} follows. \qed

\subsection{Further directions} \label{lsss}
Perhaps the universality of Theorem 1 extends
to larger subrings of $R^*(M_{g,n}^c)$. A natural
place to start is the ring
$$S^*(M_{g,n}^c) \subset R^*(M_{g,n}^c)$$
generated by all the $\kappa$ and $\psi$ classes.

\begin{Question}
Is $S^*(M_{g,n}^c)$ canonically a subring of $S^*(M_{0,2g+n}^c)$ ?
\end{Question}

At least the condition $n>0$ must be imposed in Question 2.
How to include the strata classes in a universality statement
is not clear.

\vspace{+8 pt}
\noindent
Department of Mathematics\\
Princeton University\\
rahulp@math.princeton.edu.

\end{document}